\documentclass[12pt]{article}
\usepackage[utf8]{inputenc}
\usepackage[T1]{fontenc}
\usepackage{amsmath, amssymb, amsthm}
\usepackage{color, verbatim}
\usepackage[margin = 1.2in]{geometry}
\newcommand{\E}{\mathsf{E}}
\newcommand{\var}{\mathsf{Var}}
\newcommand{\cov}{\mathrm{cov}}

\newcommand{\pr}{\mathsf{P}}
\newcommand{\dd}{\mathrm{d}}

\newcommand{\veps}{\varepsilon}

\newcommand{\ipto}{\overset{\pr}{\to}}
\newcommand{\tsp}{\mathsf{T}}

\newcommand{\rN}{\mathcal{N}}
\newcommand{\emin}{\lambda_{\min}}

\theoremstyle{plain}%
\newtheorem{thm}{Theorem}[section]
\newtheorem{lem}[thm]{Lemma}
\newtheorem{prop}[thm]{Proposition}

\theoremstyle{definition}
\newtheorem{defn}{Definition}[section]
\newtheorem{assmp}{Assumption}
\theoremstyle{remark}

\usepackage[numbers]{natbib}

\usepackage{hyperref}

\usepackage{authblk}

\begin{document}

\title{Consistent Maximum Likelihood Estimation Using Subsets with Applications to Multivariate Mixed Models}
\author{Karl Oskar Ekvall\thanks{ekvall@umn.edu}~}
\author{Galin L. Jones\thanks{galin@umn.edu}}
\affil{University of Minnesota}

\maketitle
\abstract{\noindent We present new results for consistency of maximum likelihood estimators with a focus on multivariate mixed models. Our theory builds on the idea of using subsets of the full data to establish consistency of estimators based on the full data. It requires neither that the data consist of independent observations, nor that the observations can be modeled as a stationary stochastic process. Compared to existing asymptotic theory using the idea of subsets we substantially weaken the assumptions, bringing them closer to what suffices in classical settings. We apply our theory in two multivariate mixed models for which it was unknown whether maximum likelihood estimators are consistent. The models we consider have non-stochastic predictors and multivariate responses which are possibly mixed-type (some discrete and some continuous).
}
\newpage
\section{Introduction} \label{sec:intro}
Mixed models are frequently used in applications and have been the subject of numerous articles and books \citep{Demidenko2013, Jiang2007, McCullochSearle2008}. Yet, it was unknown until recently whether maximum likelihood estimators (MLEs) are consistent even in some simple generalized linear mixed models (GLMMs) \citep{Jiang2017}. What complicates proving consistency in some mixed models is the dependence among response variables induced by certain random effect designs. Of course, not all types of dependence between responses are problematic -- there is a vast literature on maximum likelihood estimation with dependent observations \citep{Bar-Shalom1971, Crowder1976, Hall1980, Heijmans1986, Silvey1961, Wald1948, Weiss1971}. But, as we will discuss in more detail below, for some commonly used random effect designs such as those with crossed random effects, existing conditions for consistency of MLEs are hard to verify \citep{Jiang2017}. In a few GLMMs with crossed random effects, consistency has been proved using a novel argument that relates the likelihood for the full data to that of a subset consisting of independent and identically distributed (i.i.d.) random variables, ``the subset argument'' \citep{Jiang2013}.

Fundamentally, however, the issue is not unique to GLMMs or even mixed models; any other parametric model appropriate for the same settings may present similar difficulties. Accordingly, it was recognized in the first work on consistency using subsets that the idea has the potential to be extended to more general models \citep{Jiang2013}. We address this by establishing weaker conditions, based in part on the use of subsets, that are sufficient for consistency of MLEs, without assuming a particular model. They help explain formally what makes the subset argument work, why it is useful in some settings where more classical ones are not, and when it can fail. We illustrate the usefulness of our conditions by proving consistency of MLEs in two multivariate GLMMs (MGLMMs) to which existing theory has not been applied successfully.

To fix ideas, let $\Theta$ denote a parameter set, $f^n_\theta$ a joint density for the random vector $Y = (Y_1, \dots, Y_n)$, and $\theta^0$ the ``true'' parameter. Let also $L_n(\theta; Y) = f^n_\theta(Y) / f^n_{\theta^0}(Y)$ and $\Lambda_n(\theta; Y) = \log L_n(\theta; Y)$. If $\Theta$ is a finite set, then since $L_n(\theta^0; Y) = 1$, a necessary and sufficient condition for consistency of MLEs is that, as $n \to \infty$,
\begin{align} \label{eq:p_lik_zero}
 \pr(L_n(\theta; Y) \geq 1) \to 0 \quad \text{for all}\enskip \theta \neq \theta^0.
\end{align}
When $\Theta$ is not a finite set, \eqref{eq:p_lik_zero} needs to be amended by a uniformity argument to be sufficient, but the main ideas are the same. There are many ways to establish \eqref{eq:p_lik_zero}. With i.i.d. observations and regularity conditions, \eqref{eq:p_lik_zero} or stronger results follow from the law of large numbers applied to \(n^{-1}\Lambda_n(\theta; Y)\) \citep{Cramer1946,Doob1934,Ferguson1996,Wald1949}. If $Y$ is a more general stochastic process, $\Lambda_n(\theta; Y)$ may still, suitably scaled, satisfy an ergodic theorem, leading again to \eqref{eq:p_lik_zero} under regularity conditions. In the literature on maximum likelihood estimation with dependent observations, it is often assumed that some such limit law holds, either for $\Lambda_n(\theta; Y)$ or its derivatives \citep{Crowder1976,Hall1980,Heijmans1986}, or that the moments of $\Lambda_n(\theta; Y)$ converge in an appropriate way \citep{Bar-Shalom1971,Silvey1961}. Unfortunately, in many practically relevant settings, it is not clear that any such convergence holds and proving that it does is arguably the main obstacle to establishing consistency of MLEs. Let us illustrate using an MGLMM, commonly considered both in statistics and applied sciences \citep{ChenWehrly2014, CoullAgresti2000, GueorguievaAgresti2001, McCulloch2008,Warton2015}.

Let \(X = {[x_1, \dots, x_n]}^\tsp \in \mathbb R^{n \times p}\) be a matrix of non-stochastic predictors, \(Z = {[z_1, \dots, z_n]}^\tsp \in \mathbb R^{n \times r}\) a non-stochastic design matrix, and \(U \in \mathbb R^r\) a multivariate normal vector of random effects, with mean zero and covariance matrix \(\Sigma\). For the MGLMM, \( \Theta \subseteq \mathbb R^d \), for some \(d \geq 1\), \( \beta = \beta(\theta) \), and \(\Sigma = \Sigma(\theta)\). The responses \(Y_1, \dots, Y_n\) are conditionally independent given \(U\), with conditional exponential family densities
\begin{align*}
 f_{\theta, i}(y_i \mid u) = k_i(y_i, \tau_i)\exp\left(\frac{y_i [x_i^\tsp \beta + z_i^\tsp u] - c_i(x_i^\tsp \beta + z_i^\tsp u)}{\tau_i}\right),
\end{align*}
where, for $i = 1, \dots, n$, $c_i$ is the conditional cumulant function, $\tau_i$ a dispersion parameter, and $k_i(y_i, \tau_i)$ ensures $f_{\theta, i}(y_i\mid u)$ integrates to one. Conditional independence implies $f_\theta^n(y\mid u) = \prod_{i = 1}^n f_{\theta, i}(y_i\mid u)$. Several of the responses could be from the same subject, hence the ``multivariate'', and they can be of mixed type, some continuous and some discrete, for example.

The dependence among the linear predictors is easily characterized since $X\beta + ZU \sim \rN(X\beta, Z \Sigma Z^\tsp)$. The relevant density for maximum likelihood estimation, however, is the marginal density,
\begin{align} \label{eq:joint_dens}
 f_\theta^n(y) = \int_{\mathbb R^r} f_\theta^n(y \mid u)\phi^r_\theta(u)\dd u,
\end{align}
where $\phi_\theta^r$ denotes the $r$-dimensional multivariate normal density with mean zero and covariance matrix $\Sigma = \Sigma(\theta)$. The density $f^n_\theta(y)$ typically does not admit a closed form expression. Moreover, the dependence among responses it implies is in general less transparent than that among the linear predictors. What we can say in general is that two responses are dependent only if their corresponding linear predictors are. That is, response component $i$ and $j$ are independent if $z_i^\tsp \Sigma z_j = 0$.

It is convenient if $Z \Sigma Z^\tsp$ is, upon possible reordering of the responses, block diagonal since in that case the full vector of responses can be partitioned into independent sub-vectors. If these are of fixed length as $n$ grows then one is back in the classical setting where the full data consists only of an increasing number of independent vectors. This setting is common to many articles on asymptotic theory in mixed models \citep{Guven1995,Min2007,Nie2006,Sung2007}. Unfortunately, in applications the number of independent response vectors -- the number of diagonal blocks in $Z\Sigma Z^\tsp$ -- is often small. For example, Sung and Geyer \citep{Sung2007} note that in the famous salamander data \citep{McCullagh1989} there are 3 independent vectors, each of length 120. Thus, in their notation there are $n = 3$ independent observations, but in our notation there are $n = 3 \times 120 = 360$ possibly dependent observations. It seems more reasonable, then, {to consider large sample properties that do not assume the response vector $Y$ consists only of a large number of independent sub-vectors}. The type of limiting process we consider has, in the context of mixed models, previously only been investigated carefully in special cases that do not allow for predictors or  mixed-type responses \citep{Jiang2013,Miller1977}. To be sure, Jiang's \citep{Jiang2013} general theory does allow for predictors, but the specific applications do not. {Due to the inclusion of non-stochastic predictors, uniform convergence results needed for our theory, which in simpler models can be established using classical laws of large numbers, are in one of our applications verified using empirical process theory.}

The intuition behind the usefulness of the subset argument can be understood by considering the following simple LMM with crossed random effects. Suppose $Y_{i, j} = \theta + U^{(1)}_i + U^{(2)}_j + E_{i, j}$, where $U^{(1)}_i$, $U^{(2)}_j$, and $E_{i, j}$ are all i.i.d. standard normal, $i = 1, \dots, N$, $j = 1, \dots, N$. It is easy to check that the $Y_{i, j}$s cannot be partitioned into independent subsets. However, there are many subsets that, even though there is dependence among them, consist of independent random variables. For example, the two subsets $(Y_{1, 1}, Y_{2, 2}, \dots, Y_{N, N})$ and $(Y_{1, 2}, Y_{2, 3}, \dots, Y_{N - 1, N})$ are dependent, but taken separately they both consist of i.i.d. random variables. The MLE of $\theta$ based on either subset, i.e. a subset sample mean, is consistent as $N \to \infty$. Intuitively, then, the MLE based on all of the $N^2$ variables should be too. Of course, the subset argument is not needed to prove that in this simple example, but the intuition is the same for models where a direct proof is harder. How to formalize this intuition in more general models, without actually having to require the subset components to be either independent or identically distributed, is the topic of Section \ref{sec:consist}.

After developing a general theory, we will return to mixed models in Section \ref{sec:appl} and prove consistency of MLEs in two MGLMMs. The first is a longitudinal linear mixed model with autoregressive temporal dependence and crossed random effects. In this model, the integral in \eqref{eq:joint_dens} has a closed form solution which makes it easier to demonstrate some of the intricacies of the subset argument. For the second MGLMM we consider, which includes both binary and continuous responses, $f_\theta^n(y)$ does not admit a closed form expression. The subset argument is especially useful in this setting since the considered subsets have likelihoods that are more amenable to analysis than that of the full data.

The rest of the paper is organized as follows. We develop theory using subsets in Section \ref{sec:consist}. In Section \ref{sec:appl} we apply the theory from Section \ref{sec:consist} to two MGLMMs. Section \ref{sec:disc} contains a brief discussion of our results. Many technical details are deferred to the Appendix and the supplementary material \citep{Ekvall2019_suppl}.

\section{Consistency using subsets of the full data}\label{sec:consist}
Recall that $Y = (Y_1, \dots, Y_n)$ denotes a collection of random variables and let \(W = (W_1,\allowbreak \dots, W_m)\) be a collection of random variables that form a subset of those in $Y$, i.e. $\{W_1, \dots, W_m\} \subseteq \{Y_1, \dots, Y_n\}$. We will henceforth call $W$ a subcollection of $Y$ to avoid confusion with other subsets introduced later. The main results in this section give conditions for when subcollections can be used to prove consistency of maximizers of $L_n(\theta; Y)$. Unless otherwise noted, all convergence statements are as $n$ tends to infinity and the number of elements in a subcollection, $m = m(n)$, tends to infinity as a function of $n$.

All discussed random variables are defined on an underlying probability space $(\Omega, \mathcal F, \pr)$, with the elements of $\Omega$ denoted $\omega$. The parameter set $\Theta$ is assumed to be a subset of a metric space $(\mathcal T, d_{\mathcal T})$. We write, for any $t \in \mathcal T$ and $\delta > 0$, $B_\delta(t) = \{t' \in \mathcal T: d_\mathcal T(t, t') < \delta\}$. For any $A \subseteq \mathcal T$, $\bar{A}$ denotes its closure and $\partial A$ its boundary. We assume the true parameter $\theta^0$ is the same for all $n$ but the joint density $f_\theta^n(y)$ of $Y$, against a dominating, $\sigma$-finite product measure $\nu = \nu_n$, can depend on $n$ in an arbitrary manner. In particular, our setting allows for a triangular array of responses, $Y_{n, 1}, \dots, Y_{n, n}$, though for convenience we do not make this explicit in the notation.

By $\theta^0$ being the true parameter we mean that $\pr(Y \in A) = \int_A f^n_{\theta^0}(y)\nu(\dd y)$ for any measurable $A$ in the range space of $Y$. That is, expectations and probabilities with respect to $\pr$ are the same as those taken with respect to distributions indexed by $\theta^0$. Densities for the subcollection and its components are denoted by $g$ in place of $f$; for example, $L_m(\theta; W) = g_\theta^m(W) / g_{\theta^0}^m(W)$.

We will establish the following sufficient condition for consistency of maximizers of $L_n(\theta; Y)$:
\begin{align} \label{eq:p_lik_zero_unif}
 \pr\left(\sup_{\theta \in \Theta \cap B_\veps(\theta^0)^c} L_n(\theta; Y) \geq 1\right) \to 0,  ~~~~ \forall ~\veps >0\, .
\end{align}
{That is, the probability that there exists a maximizer of the likelihood outside an arbitrarily small ball around the true parameter tends to zero. We now discuss the use of subcollections and the assumptions used to achieve \eqref{eq:p_lik_zero_unif}, which eventually leads to the main results in Theorems \ref{thm:wald} and \ref{thm:cramer} presented at the end of the section.}

The appeal of using subcollections to prove \eqref{eq:p_lik_zero_unif}, instead of directly working with the full data likelihood $L_n(\theta; Y)$, can be explained using the following lemma.
\begin{lem} \label{lem:subset}
 For every $c \in (0, \infty)$, $\theta \in \Theta$, and subcollection $W$, $\pr$-almost surely,
 \[
  \pr \left( L_n(\theta; Y) \geq c \mid W\right) \leq  c^{-1} L_m(\theta; W).
 \]
\end{lem}
Versions of Lemma \ref{lem:subset} are well known \citep{Jiang2013,Jiang2017}, but the supplementary material \citep{Ekvall2019_suppl} contains a proof for completeness. From the lemma it follows that if $L_m(\theta; W) \to 0$, then $\E[\pr(L_n(\theta; Y) \geq 1 \mid W)] = \pr(L_n(\theta; Y)\geq 1) \to 0$ by dominated convergence. That is, up to a uniformity argument, $\eqref{eq:p_lik_zero_unif}$ can be established by showing that the likelihood of the subcollection converges to zero in probability, outside of a neighborhood of $\theta^0$. Uniform versions of that convergence will play a crucial role in our results.

\begin{defn}\label{def:identify}
 We say that a subset $A \subseteq \Theta$ is identified by a subcollection $W$ if $\sup_{\theta \in A}L_m(\theta; W) \ipto 0$. If $\sup_{\theta \in A}L_m(\theta; W) = O_\pr(a_n)$ for some sequence of constants $\{a_n\}$, $n = 1, 2, \dots$, we call $a_n$ an identification rate.
\end{defn}

To understand this definition better, consider the case where the subcollection $W$ consists of $m$ i.i.d. random variables with common marginal density $g_{\theta, 1}$. Suppose also that there is no $\theta \in A$ for which $g_{\theta, 1} = g_{\theta^0, 1}$ $\nu$-almost everywhere. That is, $\theta^0$ is an identified parameter in the classical sense if we restrict attention to the parameter set $A \cup\{\theta^0\}$. Then, under regularity conditions \citep[Theorems 16 and 17]{Ferguson1996}, one has $\sup_{\theta \in A}\E[\Lambda_m(\theta; W)] < 0$ and, by a uniform strong law of large numbers,
\[
  \lim_{m \to \infty} m^{-1}\sup_{\theta \in A}\vert \Lambda_m(\theta; W) - \E[\Lambda_m(\theta; W)]\vert = 0.
\]
Using this, it is straightforward to show that $A$ is identified by $W$ with an identification rate that is exponentially fast in $m$. That is, with i.i.d. components and regularity conditions, the classical definition of an identified parameter implies identification in the sense of Definition \ref{def:identify}. However, we want to allow for subcollections that do not consist of i.i.d. components, and in that case the classical definition is not as useful. For example, we have independent but not identically distributed components in one of our MGLMMs. In this and more general cases, a parameter could be identified in the classical sense for all sample sizes $n$, but, loosely speaking, the difference between the distributions for $W$ indexed by some $\theta \in A$ and that indexed by $\theta^0$ could vanish asymptotically, preventing $W$ from identifying $A$ in our sense. Finally, notice also that $A$ being identified by $W$ is essentially equivalent to MLEs based on $W$ with the restricted parameter set $A\cup \{\theta^0\}$ being consistent.

We can now be more precise about how to use subcollections to establish \eqref{eq:p_lik_zero_unif}. The strategy is to first find a subcollection $W$ that identifies $B_\veps(\theta^0)^c\cap \Theta$ for every $\veps > 0$, and then use Lemma \ref{lem:subset} to get the convergence for the full likelihood in \eqref{eq:p_lik_zero_unif}. For this strategy to be useful, showing that $W$ identifies $B_\veps(\theta^0)^c\cap \Theta$ has to be easier than showing that $Y$ does since the latter would directly imply \eqref{eq:p_lik_zero_unif}. That is, one has to be able to pick out a subcollection with more convenient properties than the full data. Our applications in Section \ref{sec:appl} illustrate how this can be done.

It is useful to allow for several subcollections $W^{(i)}$, consisting of $m_i$ components, and subsets $A_i$, $i = 1, \dots, s$. By doing so, different subcollections can be used to identify different subsets of the parameter set. For example, if the parameter set is a product space, as is common in applications, then different subcollections can be used to, loosely speaking, identify different elements of the parameter vector. Assumption \ref{A:identify} makes precise what we need to identify $\Theta \cap B_\veps(\theta^0)^c$ using several subcollections.  \begin{assmp}\label{A:identify}
 For every small enough $\veps > 0$, there are subsets $A_i = A_i(\veps) \subseteq \Theta$ and corresponding subcollections $W^{(i)}$, $i = 1, \dots, s$, such that $\cup_{i = 1}^s A_i \supseteq \Theta \cap B_\veps(\theta^0)^c$ and each $A_i$ is identified by $W^{(i)}$ with some identification rate $a_{n,i}$, $n = 1, 2, \dots$,  $i = 1, \dots, s$.
\end{assmp}
This assumption is somewhat similar to assumptions A2 and A3 made by Jiang \citep{Jiang2013}, which are also assumptions about parameter identification using several subcollections. However, those assumptions are stated in terms of $\E (\Lambda_{m_i}(\theta; W^{(i)}))$ and $\var(\Lambda_{m_i}(\theta; W^{(i)})), i = 1, \dots, s$. The fact that we do not have to assume anything about the variances of the log-likelihood ratios is an important improvement. For example, if subcollection $i$ consists of i.i.d. components, the convergence of $m_i^{-1}\Lambda_{m_i}(\theta; W^{(i)})$ is immediate from the law of large numbers, but calculating its variance may be difficult.

For finite parameter sets, Assumption \eqref{A:identify} is enough to give consistency of MLEs via Lemma \ref{lem:subset}. For more general cases we also need to control the regularity of the log-likelihood for the full data. The following two assumptions are made to ensure that the uniformity of the convergence detailed in Assumption \ref{A:identify} and Definition \ref{def:identify} carries over to $\Lambda_n(\theta; Y)$, in the sense of \eqref{eq:p_lik_zero_unif}.

\begin{assmp}\label{A:Lipschitz}
 For every $i \in \{1, \dots, s\}$ and $n \in \{ 1, 2, \dots\}$, $\Lambda_n(\theta; Y)$ is $\pr$-almost surely Lipschitz continuous in $\theta$ on the $A_i$ defined in Assumption 1; that is, there exists a random variable $K_{n, i}$ not depending on $\theta$ such that, $\pr$-almost surely and for every $\theta, \theta' \in A_i$,
 \begin{align*}
  \vert \Lambda_n(\theta; Y) - \Lambda_n(\theta'; Y)\vert \leq K_{n, i} d_\mathcal T(\theta, \theta').
 \end{align*}
\end{assmp}

\begin{assmp} \label{A:rates}
 Each $A_i$ from Assumption \ref{A:identify} can be covered by $M_{n, i}$ balls of radius $\delta_{n, i}$ such that
 \[K_{n, i} \delta_{n, i} \ipto 0\text{ and } M_{n, i}a_{n, i} \to 0,\]
 where $a_{n, i}$ and $K_{n, i}$, $i = 1, \dots, s$, $n = 1, 2, \dots,$ are the same as in Assumptions $\ref{A:identify}$ and \ref{A:Lipschitz}, respectively.
\end{assmp}

There is an interplay between Assumption \ref{A:identify} -- \ref{A:rates} where the rates in Assumption \ref{A:identify} need to be sufficiently fast in comparison to the growth of the Lipschitz constants in Assumption \ref{A:Lipschitz}; Assumption \ref{A:rates} specifies how the rates should align. Why these rates work will be clear from the proof of Lemma \ref{lem:consist}, but the intuition is as follows. To get uniformity in $\theta$, we cover $A_1$ (say) with balls small enough that the likelihood is approximately constant on them, so that one can work pointwise in $\theta$ in each ball instead of uniformly. If the likelihood changes much on $A_1$ in the sense that $K_{n, 1}$ is large, then the radius $\delta_{n, 1}$ needs to be small; this is what the first rate condition says. The second rate condition illustrates that there is a price for picking small balls, namely that if many balls are needed to cover $A_1$, then fast identification rates are needed.

The assumptions give us the convergence in \eqref{eq:p_lik_zero_unif} and, consequently, the following lemma.
\begin{lem} \label{lem:consist}
 If Assumptions \ref{A:identify} -- \ref{A:rates} hold, then the probability that there exists a global maximizer of $\Lambda_n(\theta; Y)$ in $B_\veps(\theta^0)^c \cap \Theta$ tends to zero as $n \to \infty$, for every $\veps > 0$.
\end{lem}

\begin{proof}[Proof]
  We give an outline here and a detailed proof in Appendix \ref{app:proofs}. Without loss of generality, we may assume $s = 1$, so there is one subcollection $W$ that identifies $A = \Theta \cap B_\veps(\theta^0)^c$, for arbitrary, small $\veps > 0$, with rate $a_n$. It suffices to prove that $\pr(\sup_{\theta \in A} L_n(\theta; Y) \geq 1)\to 0$. For $j = 1, \dots, M_{n}$ let $\theta^j$ be a point in the intersection of $A$ and the $j$th ball in the cover of $A$ given by Assumption \ref{A:rates}. Some algebra and Assumption \ref{A:Lipschitz} gives
 \begin{align*}
  \pr\left(\sup_{\theta \in A}L_n(\theta; Y)\geq 1 \right) \leq \pr\left(\max_{j \leq M_n} L_n(\theta^j; Y)\geq 1/2\right) + \pr\left(e^{K_n \delta_n }\geq 2 \right).
 \end{align*}
 The second term is $o(1)$ by Assumption \ref{A:rates}. It remains to deal with the first. By conditioning on the subcollection and using Lemma \ref{lem:subset} one gets
 \begin{align*}
  \pr\left(\max_{j \leq M_n} L_n(\theta^j; Y)\geq 1/2\mid W\right) & \leq 2 M_{n}\sup_{\theta \in A} L_m(\theta; W).
 \end{align*}
 The right hand side is $o_\pr(1)$ by Assumption \ref{A:rates}, so the expectation of the left hand side is $o(1)$ by dominated convergence, which finishes the proof.
\end{proof}

We will use Lemma \ref{lem:consist} to establish both a Wald-type consistency, meaning consistency of sequences of global maximizers of $L_n(\theta; Y)$, and a Cramér-type consistency, meaning consistency of a sequence of roots to the likelihood equations $\nabla \Lambda_n(\theta; Y) = 0$. It follows almost immediately from the lemma that if $L_n(\theta; Y)$ has a global maximizer $\hat{\theta}_n$, $\pr$-almost surely for every $n$, then $\hat{\theta}_n \ipto \theta^0$. In particular, if $\Theta$ is compact one gets Wald-type consistency with an additional continuity assumption. Since Assumption 2 implies $L_n(\theta; Y)$ is continuous at every point except possibly $\theta^0$, assuming continuity also at the unknown $\theta^0$ should be insignificant in any application of interest.
\begin{thm}\label{thm:wald}
 If $\Theta$ is compact, $L_n(\theta; Y)$ is $\pr$-almost surely continuous on $\Theta$ for every $n$, and Assumptions \ref{A:identify} -- \ref{A:rates} hold, then a maximizer $\hat{\theta}_n$ of $L_n(\theta; Y)$ exists $\pr$-almost surely for every $n$, and $\hat{\theta}_n \ipto \theta^0$ for any sequence of such maximizers.
\end{thm}
\begin{proof}
 Since continuous functions attain their suprema on compact sets, $L_n(\theta; Y)$ has a maximizer on $\Theta$, $\pr$-almost surely. By Lemma \ref{lem:consist} all maximizers are in $B_\veps(\theta^0)$ with probability tending to one, for all small enough $\veps > 0$.
\end{proof}
Though compactness is a common assumption \citep{Heijmans1986,White1982}, it is sometimes too restrictive or even unnecessary. If $L_n(\theta; Y)$, or more commonly $\Lambda_n(\theta; Y)$, is strictly concave in $\theta$ on a convex $\Theta$, then it is enough to verify the assumptions on a neighborhood of $\theta^0$ (c.f. Theorem \ref{thm:cramer}) to get consistency of the unique global maximizer. However, a global maximizer need not exist even as $n \to \infty$, or perhaps the assumptions cannot be verified for other reasons. With a few additional assumptions, Lemma \ref{lem:consist} can then be used to get the weaker Cramér-type consistency, which also only requires verifying assumptions for neighborhoods of $\theta^0$.
\begin{thm}\label{thm:cramer}
 If $\Theta \subseteq \mathbb R^d$ for some $d\geq 1$, $L_n(\theta; Y)$ is almost surely differentiable in $\theta$ on a neighborhood of an interior $\theta^0$ for every $n$, and Assumptions \ref{A:identify} -- \ref{A:rates} hold with $\Theta$ replaced by $\bar{B}_\veps(\theta^0)$ for all small enough $\veps > 0$, then, with probability tending to one as $n \to \infty$, there exists a local maximizer of $L_n(\theta; Y)$, and hence a root to the likelihood equation $\nabla \Lambda_n(\theta; Y) = 0$, in $B_\veps(\theta^0)$, for all small enough $\veps > 0$.
\end{thm}

\begin{proof}
 Since $\theta^0$ is interior we may assume $\veps > 0$ is small enough that all points of $\bar{B}_\veps(\theta^0)$ are interior. Almost sure differentiability of $L_n(\theta; Y)$ implies almost sure continuity. Thus, $L_n(\theta ; Y)$ attains a local maximum on the compact $\bar{B}_\veps(\theta^0)$, $\pr$-almost surely. By Lemma \ref{lem:consist}, with probability tending to one, there are no such maximizers in $\bar{B}_\veps(\theta^0)\setminus B_\veps(\theta^0) = \partial B_\veps(\theta^0)$. Thus, with probability tending to one, there exists a local maximizer in $B_\veps(\theta^0)$. Since $L_n(\theta; Y)$ and hence $\Lambda_n(\theta; Y)$ is $\pr$-almost surely differentiable, any such maximizer must be a root to the likelihood equation $\nabla \Lambda_n(\theta; Y) = 0$.
\end{proof}

In the next section we apply Theorem \ref{thm:cramer} to two special cases of the MGLMM described in Section~\ref{sec:intro}. We also discuss in more detail how to think about the subcollections and subsets in specific models.

\section{Application to multivariate mixed models} \label{sec:appl}

\subsection{Longitudinal linear mixed model} \label{sec:lmm}

The first model we consider is an extension of a variance components model that has been studied previously \citep{Miller1977}. In addition to dependence between subjects induced by crossed random effects the model incorporates autoregressive temporal dependence between measurements from the same subject. To make the discussion clearer we assume easy-to-specify fixed and random effect structures. This allows us to focus on the core issues, that is, on how to select subcollections and subsets that can be used to verify the conditions of our theory.  Our model includes a baseline mean and a treatment effect. A general fixed effect design matrix could be treated the same way as in our second example, discussed in Section \ref{sec:logit_norm}. Before establishing consistency, we discuss the model definition and how to select appropriate subcollections.

Suppose for subjects $(i, j)$, $i = 1, \dots, N$ and $j = 1, \dots, N$, and time points $t = 1, \dots, T$, we observe the response $Y_{i, j, t}$, where for convenience we assume both $N$ and $T$ are even. Let the stacked vector of responses be $$Y = [Y_{1, 1, 1}, \dots, Y_{1, 1, T}, Y_{1, 2, 1}, \dots, Y_{N, N, T}]^\tsp \in \mathbb R^{n},\ n = TN^2.$$
Recall from  the introduction that the MGLMM is specified by the conditional distribution $f_\theta^n(y \mid u)$ and the distribution of the random effects, $\phi^r_\theta(u)$. For a linear mixed model we let $f_\theta^n(y\mid u)$ be the density of a multivariate normal distribution with mean $X\beta + Zu$ and covariance matrix $\theta_3 I_n$, $\theta_3 > 0$, where the two components of $\beta = [\theta_1, \theta_2]^\tsp \in \mathbb{R}^2$ are a baseline mean and a treatment effect, respectively, and $I_k$ denotes the $k\times k$ identity matrix. Note, in the notation of the introduction, the dispersion parameter in the conditional distribution is $\tau_i  = \theta_3$, for all $i$. We treat $\theta_3$ as a parameter to be estimated and not as known, which is otherwise common in the literature.

Let $h_n$ be a vector of zeros and ones where the $i$th element is one if it corresponds to an observation in time $t \leq T / 2$ and zero otherwise and let $1_n$ denote an $n$-vector of ones. We take $X = [1_n, h_n] \in \mathbb R^{TN^2 \times 2}$, which corresponds to a treatment being applied in the first half of the experiment. Unless $T$ is fixed, which we do not assume, this setup implies the predictors change with $n$. Indeed, as $T$ grows, a particular observation can go from being made in the latter half of the experiment to the earlier half. Thus, the responses form a triangular array.

Partition $U$ into three independent sub-vectors, $U^{(1)} \sim \rN(0, \theta_4 I_N)$, $U^{(2)}\allowbreak \sim \rN(0, \theta_5 I_N)$, and $U^{(3)} \sim \rN(0, \theta_6 I_{N^2} \otimes \Psi)$, where $\Psi = (\Psi_{i, j}) = (\theta_7^{\vert i - j\vert})$ is a first order autoregressive correlation matrix, $\theta_i > 0$, $i = 4, 5, 6$, and $\theta_7 \in (-1, 1)$. We will use $U^{(1)}$ and $U^{(2)}$ as crossed random effects, inducing dependence between subjects, and $U^{(3)}$ to get temporal dependence within subjects. To that end, let $Z_1 = I_N \otimes 1_N \otimes 1_T$, $Z_2 = 1_N \otimes I_N \otimes 1_T$, and $Z = [Z_1, Z_2, I_{TN^2}]$. Then, with $J_k = 1_k1_k^\tsp$, the covariance matrix of the linear predictors $X\beta + ZU$ is
\begin{align*}
 Z \Sigma Z^\tsp & = \theta_4 I_N \otimes J_{NT} + \theta_5 J_N \otimes I_N \otimes J_T + \theta_6 I_{N^2} \otimes \Psi.
\end{align*}
More transparently, for the elements of $\E(Y \mid U) = X\beta + ZU$, it holds that
\begin{align*}
 \cov[\E(Y_{i, j, t}\mid U), \E(Y_{i', j', t'}\mid U)] & = \begin{cases}
  \theta_4 + \theta_5 + \theta_6 \theta_7^{\vert t - t'\vert} & i = i', j = j' \\
  \theta_4                                                    & i = i', j \neq j' \\
  \theta_5                                                    & i \neq i', j = j' \\
  0                                                           & \text{otherwise}
 \end{cases}.
\end{align*}
The marginal density $f_\theta^n(y)$ admits a closed form expression in this example. Specifically, the marginal distribution for $Y$ is multivariate normal with mean $m(\theta) = X\beta(\theta)$ and covariance matrix $C(\theta) = \theta_3 I_{TN^2} + Z \Sigma(\theta)Z^\tsp$. Note that the structure of $C(\theta)$ is similar to that of the covariance matrix of the linear predictors just discussed. In particular, there are many zeros in the covariance matrix $C(\theta)$, i.e. there are many independent observations, but $Y$ cannot be partitioned into independent vectors.

\subsubsection{Subcollection selection}
The model definitions imply that $\Theta = \mathbb R \times \mathbb R \times (0, \infty) \times (0, \infty) \times (0, \infty) \times (0, \infty) \times (-1, 1)$, a subset of $\mathbb R^{7}$, which we equip with the metric induced by the Euclidean norm $\Vert \cdot \Vert$. We write $\theta = (\theta_1, \dots, \theta_7)$.

Subcollections are selected for the purpose of verifying Assumption \ref{A:identify}. The main idea guiding selection is suggested by the fact that identification follows, under regularity conditions, if the subcollection's log-likelihood satisfies a law of large numbers. We will use $s = 2$ such subcollections and require that they together identify $\theta$ in the classical sense. By this we mean that, letting $\nu_\theta^i$ denote the distribution of subcollection $i$ implied by parameter $\theta$, \[
   \{\theta \in \Theta: \nu_\theta^1 = \nu_{\theta^0}^1\} \cap \{\theta \in \Theta: \nu_\theta^2 = \nu_{\theta^0}^2\} = \{\theta^0\}.
\]
With these properties in mind, we take $W^{(1)}$ to consist of the vectors \[W^{(1)}_i = (Y_{2i - 1, 2i - 1, 1}, Y_{2i, 2i, T}) \in \mathbb{R}^2,\quad i = 1, \dots, N/2.\] Because these vectors do not share any random effects, they are independent. In fact, they are i.i.d. multivariate normal with common mean $m_1(\theta) = [\theta_1 + \theta_2, \theta_1]^\tsp$ and common covariance matrix $C_1(\theta) = I_2 (\theta_3 + \theta_4 + \theta_5 + \theta_6)$. Clearly, $\theta_1$ and $\theta_2$ are identified in the classical sense by this subcollection, but not $\theta_3 , \dots, \theta_7$. Note that even though the predictors, and hence the distributions, do not change with $N$ for this subcollection, it is strictly speaking a triangular array unless $T$ is fixed.

To identify the remaining parameters, take $W^{(2)}$ to consist of the vectors
\[
  W^{(2)}_i = (Y_{2i-1, 2i-1, 1}, Y_{2i - 1, 2i - 1, 2}, Y_{2i - 1, 2i - 1, 3}, Y_{2i - 1 , 2i, 1}, Y_{2i, 2i - 1, 1}),
\]
$i = 1, \dots, N/2$. These are also i.i.d. multivariate normal, with common mean $m_2(\theta) = (\theta_1 + \theta_2)1_5$ and common covariance matrix
\begin{align*}
 C_2(\theta) = &
 \begin{bmatrix}
  \sum_{i = 3}^6 \theta_i & \theta_4 + \theta_5 + \theta_6\theta_7 & \theta_4 + \theta_5 + \theta_6\theta_7^2   & \theta_4   & \theta_5 \\
  \cdot  & \sum_{i = 3}^6 \theta_i  & \theta_4 + \theta_5 + \theta_6\theta_7 & \theta_4  & \theta_5  \\
  \cdot  & \cdot & \sum_{i = 3}^6 \theta_i & \theta_4 & \theta_5   \\
  \cdot & \cdot & \cdot  & \sum_{i = 3}^6 \theta_i & 0 \\
  \cdot  & \cdot  & \cdot & \cdot & \sum_{i = 3}^6 \theta_i
 \end{bmatrix}.
\end{align*}
It is straightforward to check that $C_2(\theta) = C_2(\theta')$ implies $\theta_i = \theta_i'$, $i = 3, \dots, 7$.

In summary, the two subcollections together identify $\theta$ in the classical sense. Moreover, since both subcollections consist of i.i.d. multivariate normal vectors, their log-likelihoods satisfy a law of large numbers as $N \to \infty$. With this we are equipped to verify that Assumptions \ref{A:identify} -- \ref{A:rates} hold locally, leading to the main result of the section in Theorem~\ref{thm:lmm}.

\subsubsection{Consistency}
The purpose of this section is to verify the conditions of Theorem \ref{thm:cramer}. The interesting part of that is to check that Assumptions \ref{A:identify} -- \ref{A:rates} hold with $\Theta$ replaced by $\bar{B}_\veps(\theta^0)$, for all small enough $\veps > 0$. For this purpose we will first prove two lemmas that roughly correspond to Assumptions \ref{A:identify} and \ref{A:Lipschitz}. The limiting process we consider is that $N$ tends to infinity while $T$ can be fixed or tend to infinity with $N$, at rates discussed below. Thus, the statements $n \to \infty$ and $N \to \infty$ are equivalent. We will need the following result which is proved in the supplementary material \citep{Ekvall2019_suppl}.

\begin{prop} \label{prop:id_compact}
 If $\Theta$ is compact, $L_{m_i}(\theta; w^{(i)})$ is continuous in $\theta$ on $\Theta$ for every $w^{(i)}$ in the support of $W^{(i)}$, $i = 1, \dots, s$, and $\cap_{i = 1}^s \{\theta \in \Theta: \nu_\theta^i = \nu_{\theta^0}^i\} = \{\theta^0\}$, then for any $\veps > 0$ there are compact sets $\tilde{A}_1, \dots, \tilde{A}_s$ such that $\{\theta \in \Theta: \nu_\theta^i = \nu_{\theta^0}^i\} \cap \tilde{A}_i = \emptyset$, $i = 1, \dots, s$, and $\cup_{i = 1}^s \tilde{A}_i = \Theta \cap B_\veps(\theta^0)^c$.
\end{prop}
Note, when applying the proposition in the present application, $m_i = N$, $s = 2$, and $\Theta$ is replaced by $\bar{B}_\veps(\theta^0)$. As we will see in the proof of the following lemma, the proposition is useful because the $\tilde{A}_i$s it gives are compact. Lemma \ref{lem:lmm_id} formalizes verification of Assumption \ref{A:identify}.

\begin{lem} \label{lem:lmm_id} If $\theta^0$ is an interior point of $\Theta$, then for all small enough $\veps > 0$ there exist subsets $A_1$ and $A_2$ such that $A_1 \cup A_2 = \partial B_\veps(\theta^0)$,
 \begin{enumerate}
  \item $N^{-1}\sup_{\theta \in A_i} \E [\Lambda_{N/2}(\theta; W^{(i)})] = \sup_{\theta \in A_i}\E [\Lambda_1(\theta; W^{(i)}_1)]/2 < 0$,
  \item $\pr$-almost surely, $N^{-1}\sup_{\theta \in A_i}\vert \Lambda_{N/2}(\theta; W^{(i)}) - \E[\Lambda_{N/2}(\theta; W^{(i)})] \vert \to 0$, and, consequently;
  \item $A_i$ is identified by $W^{(i)}$ with an identification rate $a_{n, i} = o(e^{-\epsilon N(n) })$ for some $\epsilon > 0$, $i = 1, 2$.
 \end{enumerate}
\end{lem}
\begin{proof}
We give an outline here and a detailed proof in the supplemental material \citep{Ekvall2019_suppl}. It is easy to check that the requirements of Proposition \ref{prop:id_compact} are satisfied with $\Theta$ replaced by $\bar{B}_\veps(\theta^0)$. By taking the $A_i$s to be the $\tilde{A}_i$s given by Proposition \ref{prop:id_compact}, proving points 1 -- 2 is similar to proving that MLEs based on subcollection $i$ are consistent if the parameter set is restricted to the compact set $A_i\cup\{\theta^0\}$, $i = 1, 2$. Since the subcollection components are i.i.d., this is straightforward using classical ideas \citep[Theorems 16 and 17]{Ferguson1996}. The only difference from the referenced work is that one subcollection is a triangular array and so we use a different strong law. Point 3 follows from points 1 and 2.
\end{proof}

Note that, in this lemma and elsewhere, $\epsilon$ is a small number that is defined in context whereas $\veps$ always denotes the radius of the neighborhood of $\theta^0$ we are considering. It remains to verify the assumptions concerned with the regularity of the log-likelihood of the full data. When the log-likelihood is differentiable, Lipschitz continuity follows from the mean value theorem if the gradient is bounded. The following lemma uses that to verify Assumption \ref{A:Lipschitz}. The resulting Lipschitz constant, i.e. the bound of the gradient, is the same for both \(A_1\) and \(A_2\). The lemma also gives a probabilistic bound on the order of this Lipschitz constant as \(n \to \infty\) that will be useful when verifying Assumption \ref{A:rates}.
\begin{lem} \label{lem:lmm_lipsh}
 If \(\theta^0\) is an interior point of \(\Theta\), then for every \(n\) and small enough \(\veps > 0\) there exists a random variable \(K_n\) such that, \(\pr\)-almost surely,
 \[
  \sup_{\theta \in \bar{B}_\veps(\theta^0)} \Vert \nabla \Lambda_n(\theta; Y)\Vert \leq K_n = o_\pr(n^b),
  \]
for some \(b > 0\).

\end{lem}
Proving Lemma \ref{lem:lmm_lipsh} (see the supplementary material \citep{Ekvall2019_suppl}) is largely an exercise in bounding the eigenvalues of the covariance matrix $C(\theta)$ and its inverse on interior points of $\Theta$. We are ready for the main result of the section.

\begin{thm} \label{thm:lmm}
 If $\theta^0$ is an interior point of $\Theta$ and $T = O(N^k)$ for some $k \geq 0$ as $N \to \infty$, then, $\pr$-almost surely, there exists a sequence $\hat{\theta}_n$ of roots to the likelihood equations $\nabla \Lambda_n(\theta; Y) = 0$ such that $\hat{\theta}_n \ipto \theta^0$.
\end{thm}

\begin{proof}
 We verify the conditions of Theorem \ref{thm:cramer}. Fix an arbitrary $\veps > 0$. Since $\theta^0$ is interior we may assume $\veps$ is small enough that all points in $\bar{B}_\veps(\theta^0)$ are interior points of $\Theta$.  As is proven in the supplementary material \citep{Ekvall2019_suppl}, $\ell_n(\theta; Y) = \log(f_\theta^n(Y))$ is $\pr$-almost surely differentiable on $\bar{B}_\veps(\theta^0)$, so $\Lambda_n(\theta; Y) = \ell_n(\theta, Y) - \ell_n(\theta^0; Y)$ is too. By Lemma \ref{lem:lmm_id}, Assumption \ref{A:identify} holds with what is there denoted $\Theta$ replaced by $\bar{B}_\veps(\theta^0)$. The identification rate is exponentially fast in $N / 2$, $a_n = o(e^{-N\epsilon})$ for some $\epsilon > 0$. Lemma \ref{lem:lmm_lipsh} shows that $\Lambda_n(\theta; Y)$ is $K_n$-Lipschitz on both $A_1$ and $A_2$, and that $K_n = o_\pr(n^b)$ for some $b > 0$. This verifies Assumption \ref{A:Lipschitz}. It remains only to verify that the rate conditions in Assumption 3 hold. The $\delta$-covering number of the sphere $ \partial B_\veps(\theta^0)$ is $O([\veps / \delta]^{d - 1})$ as $\delta \to 0$ \citep[][Lemma 1]{Bronshteyn1975}. Thus, since $A_i \subseteq \partial B_\veps(\theta^0)$, by picking $\delta_{i, n} = n^{-b}$ we can have $M_{n, i} = O(n^{[d - 1]b})$ as $n \to \infty$, $i = 1, 2$. Our choice of $\delta_{n, i}$ ensures $K_{n, i} \delta_{n, i} = K_n\delta_n = o_\pr(1)$, which is the first rate condition. Since the identification rate is exponential in $N/2$ for both subcollections, we have that $M_{n, i}a_{n, i} = O(N^{2b[d - 1]}T^{b[d - 1]}e^{-\epsilon N})$ for some $\epsilon > 0$, which is $o(1)$ as $N \to \infty$ since $T$ is of (at most) polynomial order in $N$.
\end{proof}

We expect the proof technique used here to work in many other models. Essentially, all that is needed is that the subcollections' log-likelihoods satisfy uniform strong laws, that the gradient of the full data log-likelihood is of polynomial order, and that the number of observations in any subcollection grows faster than logarithmically in the total sample size. Here, to show that the gradient of the log-likelihood is of polynomial order (Lemma \ref{lem:lmm_lipsh}) we worked with the closed form expression $2 \log f^n_\theta(Y) =  -\log \det(C(\theta)) - (Y - m(\theta))^\tsp C^{-1}(\theta)(Y - m(\theta)) - n\log(2\pi)$, using that $\Vert Y\Vert$ is of polynomial order and that the eigenvalues of $C(\theta)$ are appropriately bounded by polynomials in $n$ on $B_\varepsilon(\theta^0)$. This illustrates that, in order to determine if the gradient in a given model is of polynomial order or not, one in general has to consider both the stochastic properties of the data and the particular parameterization. Uniform strong laws for the subcollections' log-likelihoods, leading to Lemma \ref{lem:lmm_id}, hold here because the subcollections consist of i.i.d. random variables. This is clearly not necessary; in the next example we consider subcollections with independent but not identically distributed variables, and, similarly, strong laws for stationary stochastic processes may apply if one has a model with subcollections consisting of dependent but identically distributed variables.

It is possible that the assumption that $T = O(N^k)$, $k \geq 0$, could be relaxed by picking other subcollections that also make use of the variation in the time dimension. It is not trivial, however, since the dependence between any two responses sharing a random effect does not vanish as time between the observations increases. Indeed, it is crucial that $N\to \infty$ in this model: if $T\to \infty$ but $N$ is fixed, then the data consist of a fixed number ($N^2$) of vectors of $T$ equicorrelated random variables. In that case it is not possible to find a subcollection that consists of an increasing number of independent variables. Accordingly, one can show that even if one was to simplify our model so that the $N^2$ vectors were independent and one was estimating only a mean parameter, the MLE would not be consistent. In the next section we examine how predictors and mixed-type responses affect the argument.

\subsection{Logit-normal MGLMM} \label{sec:logit_norm}
The model we consider in this section is an extension in several ways of the logistic GLMMs for which the technique based on subcollections was first developed \citep{Jiang2013}. The random effect structures are similar, i.e. crossed, but we have multivariate, mixed-type responses, and predictors. The main ideas for verifying the assumptions of the theory from Section \ref{sec:consist} are the same as in our LMM example. However, due to the inclusion of predictors, we use results from empirical process theory in place of the more classical strong laws used for the LMM. Showing existence of appropriate subsets of the parameter space that the subcollections identify also requires more work than with i.i.d. components. As before, we discuss the model definition and subcollection selection before establishing consistency.

Suppose for subjects $(i, j)$, $i = 1, \dots$ and $N, j = 1, \dots, N$, there are two responses, $Y_{i, j, 1}$ which is continuous and $Y_{i, j, 2}$ which is binary. The vector of all responses is
\begin{align*}
 Y = [Y_{1, 1, 1}, Y_{1, 1, 2}, Y_{1, 2, 1}, \dots, Y_{N, N, 2}]^\tsp \in \mathbb R^n, \ n = 2N^2.
\end{align*}
For each subject we observe a vector of non-stochastic predictors $x_{i, j} \in \mathbb R^p$, the same for both responses. Similarly, $z_{i, j} \in \mathbb R^r$ is the same for both responses. Let $\eta_{i, j, k} = x_{i, j} ^\tsp \beta_k + z_{i, j}^\tsp u$ be the linear predictor, $i = 1, \dots, N$, $j = 1, \dots, N$, $k = 1, 2$, where $\beta_1 = [\theta_1, \dots, \theta_{p}]^\tsp$, $\beta_2 = [\theta_{p + 1}, \dots, \theta_{2p}]^\tsp$. We assume that $\Vert x_{i, j}\Vert \leq 1$ for all $i, j$. In practice this only rules out the possibility that $\Vert x_{i, j}\Vert  = \infty$ since our setting allows for the standardization of predictors. The conditional density of the responses given the random effects that we consider is, up to scaling by $(2\pi)^{-n/2}$,
\begin{align*}\label{eq:mglmm_dens}
 f_\theta^n(y\mid u) & \propto \exp\left[\sum_{i, j}-(y_{i, j, 1} - \eta_{i, j, 1})^2/2 + y_{i, j, 2}\eta_{i, j, 2} - \log\left(1 + e^{\eta_{i, j, 2}}\right)\right].
\end{align*}
Given the random effects, $Y_{i, j, 1}$ is normal with mean $\eta_{i, j, 1}$ and variance 1, and $Y_{i, j, 2}$ is Bernoulli with success probability $1 / (1 + e^{-\eta_{i, j, 2}})$ -- a logistic GLMM. The choice of $\tau_i = 1$ for all $i$ is made for identifiability reasons for the Bernoulli responses, and for convenience for the normal responses. Setting the $\tau_i$s to some other known constants does not fundamentally change the results.

Suppose $U^{(1)} \sim \rN(0, \theta_{d} I_N)$ and $U^{(2)} \sim \rN(0, \theta_{d} I_N)$, independently, with corresponding design matrices $Z_1 = I_N \otimes 1_N \otimes 1_2$ and $Z_2 = 1_N \otimes I_n \otimes 1_2$. Taking $U = [U^{(1) \tsp}, U^{(2) \tsp}]^\tsp$ and $Z = [Z_1, Z_2]$ the linear predictors are $\eta_{i, j, k} = x_{i, j}^\tsp \beta_k + u^{(1)}_i + u_j^{(2)}$. Thus, responses from the same subject share two random effects, responses from different subjects with one of the first two indexes in common share one random effect, and other responses share no random effects and are hence independent. The covariance matrix for the linear predictors is easily computed in the same way as in the LMM. The covariance matrix for responses, however, is less transparent. It is for simplicity that we assume in this section that all random effects have the same variance. It is not necessary for our theory to be operational but this simplification shortens proofs considerably and allows us to focus on the main ideas.

\subsubsection{Subcollection selection}
With $p$ predictors the $(2p + 1)$-dimensional parameter set is $\Theta = \mathbb R^p \times \mathbb R^p \times (0, \infty)$, a subset of $\mathbb R^d$, again equipped with the usual Euclidean metric. The intuition behind the selection of subcollections is that the normal responses should identify the coefficient $\beta_1$ and the variance parameter $\theta_d $. Similarly, the Bernoulli responses should identify the coefficient vector $\beta_2$. With that in mind we take, for $i = 1, 2$,
 \[W^{(i)} = (Y_{1, 1, i}, Y_{2, 2, i}, \dots, Y_{N, N, i})\]
Both of these subcollections consist of independent but not identically distributed random variables -- independence follows from the fact that no components in the same subcollection share random effects. Notice that these subcollections are in practice often triangular arrays since the predictors may need to be scaled by $1 / \max_{i \leq N, j \leq N}\Vert x_{i, j}\Vert$ to satisfy $\Vert x_{i, j}\Vert \leq 1$. All responses in the first subcollection have marginal normal distributions and all responses in the second have marginal Bernoulli distributions.

Identification is more complicated than in our previous example. One issue is that there can be many $\theta_d$ and $\beta_2$ that give the same marginal success probability for the components in the second subcollection. A second issue is that, since the predictors can change with $n$, classical identification for a fixed $n$ does not necessarily lead to identification in the sense of Definition \ref{def:identify}.  Additionally, the approach used in the LMM to find appropriate subsets $A_1$ and $A_2$ by means of Proposition \ref{prop:id_compact} only works in general when the subcollection components are i.i.d. Thus, we take a slightly different route to establishing consistency compared to the LMM.

\subsubsection{Consistency}
In this section we verify the conditions of Theorem \ref{thm:cramer}. The limiting process is that $N \to \infty$, which is equivalent to $n \to \infty$ since $n = 2N^2$. We will first prove two lemmas that roughly correspond to Assumptions \ref{A:identify} and \ref{A:Lipschitz}.

Let $\emin(\cdot)$ denote the minimum eigenvalue of its matrix argument.
\begin{lem} \label{lem:mglmm_id}
 If $\theta^0$ is an interior point of $\Theta$ and
 \begin{equation*}
  \liminf_{N \to \infty} \emin\left(N^{-1}\sum_{i = 1}^N x_{i, i}x_{i, i}^\tsp \right) > 0,
 \end{equation*}
 then for all small enough $\veps > 0$ there exist $A_1$ and $A_2$ such that $A_1 \cup A_2 = \partial B_\veps(\theta^0)$,
 \begin{enumerate}
  \item $\limsup_{N \to \infty}  N^{-1} \sup_{\theta \in A_i}\E [\Lambda_N(\theta; W^{(i)})] < 0$,

  \item $\sup_{\theta \in A_i} N^{-1}\left\vert \Lambda_N(\theta; W^{(i)}) - \E[ \Lambda_N(\theta; W^{(i)})]\right\vert \ipto 0$, and, consequently;

  \item $A_i$ is identified by $W^{(i)}$ with an identification rate $a_{n, i} = o(e^{-\epsilon N })$ for some $\epsilon > 0$, $i = 1, 2$.
 \end{enumerate}
\end{lem}
\begin{proof}
 A detailed proof is Appendix \ref{app:proofs}, we here give the proof idea. Let $A_2 = \partial B_\veps(\theta^0)\cap \{\theta : \vert \theta_d - \theta_d^0\vert \leq \zeta\}\cap\{\Vert \beta_2 - \beta_2^0\Vert \geq \veps /2\}$, for some small $\zeta > 0$. Let $A_1$ be the closure of $\partial B_\veps(\theta^0) \cap A_2^c$. The idea is that if $\zeta$ is small enough, so that $\theta_d \approx \theta_d^0$ and $\Vert \beta_2 - \beta_2^0 \Vert \geq \veps /2$ on $A_2$, then the distributions of $W^{(2)}$ implied by $\theta \in A_2$ and $\theta^0$ are different if $X = [x_{1, 1}, x_{2, 2}, \dots, x_{N, N}]^\tsp$ has full column rank. That is, $W^{(2)}$ should be able to distinguish every $\theta \in A_2$ from $\theta^0$. Moreover, one can show that on $A_1$ it holds either that $\vert \theta_d - \theta_d^0 \vert \geq \min(\zeta, \veps/4)$ or that $\Vert \beta_1 - \beta_1^0\Vert \geq \veps /4$. In either case, $W^{(1)}$ should be able to distinguish $\theta \in A_1$ from $\theta^0$. Formalizing this idea leads to point 1. Point 2 follows from checking the conditions of a uniform law of large numbers \citep[Theorem 8.2]{Pollard1990} and point 3 from points 1 and 2.
\end{proof}

The explicit construction of the subsets $A_1$ and $A_2$, as opposed to using Proposition \ref{prop:id_compact}, warrants an additional comment. Recall, the proposition gives compact $\tilde{A}_1$ and $\tilde{A}_2$ such that $\tilde{A}_1 \cup \tilde{A}_2 = \partial B_\veps(\theta^0)$ and $\nu_\theta^i \neq \nu_{\theta^0}^i$, $\theta \in \tilde{A}_i$, $i = 1, 2$. If one takes $A_i = \tilde{A_i}$, then point 1 in Lemma \ref{lem:lmm_id} follows. Moreover, when the subcollection components are i.i.d., this in turn leads to point 1 in Lemma \ref{lem:mglmm_id}, which is what is really needed. However, when the distributions of the subcollection components are not identical, this last implication is not true in general.

Having selected appropriate subcollections and subsets it remains only to check that the log-likelihood for the full data satisfies the regularity conditions in Assumptions \ref{A:Lipschitz} -- \ref{A:rates}. The following lemma verifies Assumption \ref{A:Lipschitz} and establishes a rate needed for the verification of Assumption \ref{A:rates}.

\begin{lem} \label{lem:mglmm_lipsh}
 If $\theta^0$ is an interior point of $\Theta$, then for every $n$ and small enough $\veps > 0$ there exists a random variable $K_n$ such that, $\pr$-almost surely,
 \[
  \sup_{\theta \in \bar{B}_\veps(\theta^0)} \Vert \nabla \Lambda_n(\theta; Y)\Vert \leq K_n = o_\pr(n^b),
 \]
 for some $b > 0$.
\end{lem}

Upon inspecting the proof (supplementary material \citep{Ekvall2019_suppl}) one sees that $b$ can be taken to be $1 + \epsilon$, for any $\epsilon > 0$. This is a better (slower) rate than that obtained in the linear mixed model (see the proof of Lemma \ref{lem:lmm_lipsh}). We are now ready to state the main result of the section.

\begin{thm} \label{thm:mglmm}
 If $\theta^0$ is an interior point of $\Theta$ and
 \begin{equation*}
  \liminf_{N \to \infty} \emin\left(N^{-1}\sum_{i = 1}^N x_{i, i}x_{i, i}^\tsp \right) > 0,
 \end{equation*} then, $\pr$-almost surely, there exists a sequence $\hat{\theta}_n$ of roots to the likelihood equations $\nabla \Lambda_n(\theta; Y) = 0$ such that $\hat{\theta}_n \ipto \theta^0$.
\end{thm}

\begin{proof}
 The proof is similar to that of Theorem \ref{thm:lmm} so we skip some details. We may assume all points in $\bar{B}_\veps(\theta^0)$ are interior points of $\Theta$.  As is proven in the supplementary material \citep{Ekvall2019_suppl}, $\Lambda_n(\theta; Y)$ is differentiable on $\bar{B}_\veps(\theta^0)$. By Lemma \ref{lem:mglmm_id}, the identification rate is exponentially fast in $N$ and Lemma \ref{lem:mglmm_lipsh} shows that $\Lambda_n(\theta; Y)$ is $K_n$-Lipschitz on both $A_1$ and $A_2$, and that $K_n = o_\pr(n^b)$ for some $b > 0$. This verifies Assumption \ref{A:Lipschitz}. By picking $\delta_{i, n} = n^{-b}$ we can have $M_{n, i} = O(n^{[d - 1]b})$ as $n \to \infty$, $i = 1, 2$. Thus, $K_n\delta_n = o_\pr(1)$ and $M_{n, i}a_{n, i} = O(N^{2b[d - 1]}e^{-\epsilon N})$ for some $\epsilon > 0$, which is $o(1)$ as $N \to \infty$ since $n = 2N^2$.
\end{proof}

\section{Discussion} \label{sec:disc}

Our theory develops the current state-of-the-art asymptotic theory based on subcollections to cover more general cases. The assumptions we make highlight what makes the use of subcollections work. In particular, the interplay between the identification rates of subcollections and the regularity of the likelihood function for the full data is made precise. We note that when the subcollections consist of $m \in \{1, 2, \dots\}$ independent random variables, as in our examples, then if $n = o(m^b)$ for some $b > 0$ and $\nabla \Lambda_n(\theta; Y) = o_\pr(n^{b'})$ for some $b' > 0$, uniformly on a compact $\Theta$, the rate conditions are satisfied. This is so because, under regularity conditions, the identification rate in a subcollection with $m$ independent random variables is exponential in $m = n^{1/b}$. Since this argument works for arbitrarily large $b$ and $b'$ our theory is operational in a wide range of models. Loosely speaking, if the score function is of less than exponential order in the sample size and there are subcollections of independent random variables that grow faster than logarithmically in the sample size, the MLE is consistent. The conditions should be verifiable in many models since they often require only standard asymptotic tools. For example, in the LMM example nothing more than a uniform law of large numbers and strict positivity of the K--L divergence between distributions corresponding to distinct, identified parameters is needed. Though not pursued here, by inspecting the assumptions of our theory one also sees that it has the potential to be extended to allow the dimension of the parameter set, $d$, grow with $n$. The rates required in our assumptions could be satisfied also if $d$ grows, at least if at a slow enough rate. The wide applicability of empirical process theory, which we use in the second application, also suggests that it may be possible to verify our conditions in yet more complicated models.

Consistency of MLEs has not previously been established in either of the two models to which we apply the general theory. In particular, previous work on asymptotic theory for MLEs in mixed models often either assumes independent replications of a response vector, that there are no predictors, or no mixed-type responses. We have tried to keep the models here as simple as possible while still illustrating key ideas. Crossed random effects, temporal dependence, and predictors are included because they are challenging theoretically and are commonly used in practice. We have refrained from including things that do not require any new methods but make ideas less transparent. For example, it would be straightforward to include random effects that are not crossed, possibly at the expense of using more subcollections or subcollections consisting of independent vectors of larger dimension than what is now necessary. Similarly, adding several crossed random effects does not make things much harder, only less transparent.

Avenues for future research includes the rate of convergence of the MLEs as well as their asymptotic distribution. Intuitively, one expects MLEs based on the full data to converge at least as fast as the slowest of the subcollection MLEs, that is, the estimators one gets from using only a subset of the full data. There is some evidence of this, namely that, under regularity conditions, the Fisher information in the full data is always larger than that in any subcollection \cite{Jiang2013}. On the other hand, it is easy to show that, for the simple LMM example in the introduction, the full data MLE converges at the same rate as that based on a subcollection of $N = \sqrt{n}$ i.i.d. observations; that is, at the rate $n^{1/4}$. Given the similarities of the random effect structures, that convergence rate may in future work be a reasonable working hypothesis for MLEs in the MGLMM considered here.

\newpage
\bibliographystyle{abbrv}
\bibliography{../subset_mglmm.bib}
\newpage

\appendix

\section{Proofs} \label{app:proofs}
\begin{proof}[Proof of Lemma \ref{lem:consist}]
 Fix some arbitrary $\veps > 0$. If $\sup_{\theta \in A_i}L_n(\theta; Y) < 1$ for $i = 1, \dots, s$, then, since $L_n(\theta^0; Y) = 1$, there are no global maximizers in $\cup_{i = 1}^s A_i \supseteq \Theta \cap B_\veps(\theta^0)^c$. Thus, it suffices to prove $$\pr\left(\bigcup_{i = 1}^s \left\{\sup_{\theta \in A_i} L_n(\theta; Y) \geq 1 \right\}\right)\leq \sum_{i = 1}^ s \pr\left(\sup_{\theta \in A_i} L_n(\theta; Y) \geq 1 \right)\to 0. $$
 Since $s$ is fixed it is enough that $\pr \left( \sup_{\theta \in A_i} L_n(\theta; Y) \geq 1\right) \to 0$ for every $i = 1, \dots, s$. Without loss of generality, consider $i = 1$. Pick a cover of $A_1$ as given by Assumption 3 and, for every ball in the cover, pick a $\theta^j$ in the intersection of that ball with $A_1$. If there are some balls that do not intersect $A_1$, they may be discarded from the cover, so we assume without loss of generality that all balls do intersect $A_1$. We then get $M_{n, 1}$ points such that every point in $A_1$ is within $\delta_{n, 1}$ of at least one of them. For any $\theta \in A_1$, let $\theta^j(\theta)$ denote the $\theta^j$ closest to it (pick an arbitrary one if there are many). Using the Lipschitz continuity given by Assumption 2 and that $x\mapsto e^x$ is increasing we have,
 \begin{align*}
  \pr\left(\sup_{\theta \in A_1} L_n(\theta; Y) \geq 1 \right) & = \pr\left(\sup_{\theta \in A_1} \Lambda_n(\theta; Y) \geq 0 \right) \\
   & = \pr\left(\sup_{\theta \in A_1} \ell_n(\theta; Y) \geq \ell_n(\theta^0; Y) \right)
 \end{align*}
 which is upper bounded by
 \begin{align*}
   \pr\left(\sup_{\theta \in A_1} \left[ \ell_n(\theta^j(\theta); Y) + K_{n, 1}d_\mathcal T(\theta, \theta^j(\theta))\right] \geq \ell_n(\theta^0; Y) \right).
 \end{align*}
 Because there are only $M_{n,1}$ points $\theta^j$, and $d_\mathcal T(\theta^j(\theta), \theta) \leq \delta_{n, 1}$ since $\theta^j(\theta)$ is the one closest to $\theta$, we get that the last inline equation is upper bounded by
 \begin{align*}
  & \pr\left(\max_{j \leq M_{n, 1}} f_{\theta^j}(Y) e^{K_{n, 1}\delta_{n, 1}} \geq f_{\theta^0}(Y)\right)\\
   &\leq \pr\left(2 \max_{j \leq M_{n, 1}} f_{\theta^j}(Y) \geq f_{\theta^0}(Y)\right) +\pr\left( e^{K_{n, 1}\delta_{n, 1}} \geq 2\right) \\
  & =  \pr\left(2 \max_{j \leq M_{n, 1}} f_{\theta^j}(Y) \geq f_{\theta^0}(Y)\right) + o(1)
 \end{align*}
 where the last line uses Assumption \ref{A:rates}. The remaining term, $$\pr\left(2 \max_{j \leq M_{n, 1}} f_{\theta^j}(Y) \geq f_{\theta^0}(Y)\right) = \pr\left(\max_{j \leq M_{n, 1}}L_n(\theta^j; Y) \geq 1/2 \right),$$ we will deal with using Lemma \ref{lem:subset} and dominated convergence. After conditioning on $W^{(1)}$ we have
 \begin{align*}
\pr\left(\max_{j \leq M_{n, 1}}L_n(\theta^j; Y) \geq 1/2 \mid W^{(1)} \right) & \leq \sum_{i = 1}^{M_{n, 1}} 2 L_{m_1}(\theta^j; W^{(1)}) \\
& \leq 2 M_{n, 1} \sup_{\theta \in A_1}L_{m_1}(\theta, W^{(1)}),
 \end{align*}
 $\pr$-almost surely, where the first inequality is by subadditivity and Lemma \ref{lem:subset}, and the second uses that $L_n(\theta^j; W^{(1)}) \leq \sup_{\theta \in A_1}L_{m_1}(\theta; W^{(1)})$ by definition. The expression in the last line vanishes as $n \to \infty$ by Assumption 3. Thus,
$$\pr\left(\max_{j \leq M_{n, 1}}L_n(\theta^j; Y)\geq 1/2 \right) \to 0$$
by dominated convergence. The dominating function can be the constant 1. This finishes the proof.
\end{proof}

Let $\mathrm{C}(\delta, G, \Vert \cdot \Vert)$ denote the $\delta$-covering number of the set $G$ under the distance associated with the norm $\Vert \cdot \Vert$. We will use the following result due to Pollard \citep[Theorem 8.2]{Pollard1990}, here stated in terms of covering numbers instead of packing numbers.

\begin{lem} \label{lem:pollard}
 Let $h_1(\omega, \theta), h_2(\omega, \theta), \dots$, $\theta \in A \subseteq \Theta$, be independent processes with integrable envelopes $H_1(\omega), H_2(\omega)$, \dots, meaning $\vert h_i(\omega, \theta)\vert \leq H_i(\omega)$, for all $i$ and $\theta \in A$. Let $H = (H_1, \dots, H_N)$ and \[\mathcal{H}_{N, \omega} = \{[h_1(\omega, \theta), \dots, h_N(\omega, \theta)] \in \mathbb R^N: \theta \in A\}.\] If for every $\epsilon > 0$ there exists a $K > 0$ such that
 \begin{enumerate}
  \item $N^{-1}\sum_{i = 1}^N \E [H_i I(H_i > K)] < \epsilon$ for all $N$, and
  \item $\log \mathrm{C}(\epsilon \Vert H \Vert_1, \mathcal H_{N, \omega}, \Vert \cdot \Vert_1) = o_\pr(N)$ as $N \to \infty$,
 \end{enumerate}
 then $$\sup_{\theta \in A} N^{-1}\left\vert\sum_{i = 1}^N h_i(\omega, \theta) - \E(h_i(\omega, \theta))\right\vert \ipto 0.$$
\end{lem}

\begin{proof}[Proof of Lemma \ref{lem:mglmm_id}]
 Let us first prove that, given $\veps > 0$, there exists a $\zeta> 0$, and hence $A_i = A_i(\veps, \zeta)$, $i = 1, 2$, such that point 1 in the lemma holds. The definition of $A_i(\veps, \zeta)$ is as in the main text. Let $c(t) = \log(1 + e^t)$ denote the cumulant function in the conditional distribution of $Y_{i, i, 2}$ given the random effects and define
 \begin{align*}
  p_i(\beta_2, \theta_d) & = \E\left[c'\left(x_{i, i}^\tsp \beta_2 + \sqrt{\theta_d / \theta_d^0}\left(U_{i}^{(1)} + U^{(2)}_j\right)\right)\right].
 \end{align*}
 Recall, $\E$ denotes expectation with respect to the distributions indexed by $\theta^0$, so $p_i(\beta_2, \theta_d)$ is the success probability of $Y_{i, i, 2}$ when $\beta_2$ and $\theta_d$ are the true parameters.

 Note that because the components in $W^{(2)}$ are independent, we can write $\E [\Lambda_N(\theta; W^{(2)})]$ as a sum of $N$ terms, each summand being the negative K--L divergence between two Bernoulli variables with parameters $p_i(\beta_2, \theta_d)$ and $p_i(\beta_2^0, \theta_d^0)$. Thus (see the supplementary material \citep{Ekvall2019_suppl}),
 \begin{align*}
  N^{-1}\E[ \Lambda_N(\theta; W^{(2)})] \leq  -2N^{-1}\sum_{i = 1}^N [p_i(\beta_2, \theta_d) - p_i(\beta_2^0, \theta_d^0)]^2
\end{align*}
which one can show is upper bounded by
\begin{align}\label{eq:KL_bound_2}
   -2\left[ N^{-1}\sum_{i = 1}^N \vert p_i(\beta_2, \theta_d) - p_i(\beta_2, \theta_d^0)\vert - N^{-1}\sum_{i = 1}^N\vert p_i(\beta_2^0, \theta_d^0) - p_i(\beta_2, \theta_d^0) \vert \right]^2.
 \end{align}
 Let us work separately with the averages in the last line. We will show that the second can be made arbitrarily small on $A_2$ by selecting $\zeta$ small enough, and that the first is bounded away from zero on the same $A_2$, leading to an asymptotic upper bound on $\sup_{\theta \in A_2}N^{-1}\E[\Lambda_N(\theta; W^{(2)})]$ away from zero. We start with the first average.

 Let $H$ be a compact subset of $\mathbb R$ such that $x_{i, i}^\tsp \beta_2 \in H$ for all $i$ and $\theta \in \bar{B}_\veps(\theta^0)$. Such $H$ exists because the predictors are bounded and $\beta_2$ is bounded on $\bar{B}_\veps(\theta^0)$. Then, defining $\tilde{p}_i(\gamma, \theta_d)$ as $p_i(\beta_2, \theta_d)$ but with $x_{i, i}^\tsp \beta_2$ replaced by $\gamma$, we get
 \begin{align*}
  \sup_{\theta \in A_2} \vert p_i(\beta_2, \theta_d) - p_i(\beta_2, \theta_d^0)\vert
   & \leq \sup_{\theta \in A_2} \sup_{\gamma \in H} \vert \tilde{p}_i(\gamma, \theta_d) - \tilde{p}_i(\gamma, \theta_d^0) \vert.
 \end{align*}
 Since the random variable in the expectation defining $\tilde{p}_i$ is bounded by 1 (it is the mean of a Bernoulli random variable), $\tilde{p}_i$ is continuous by dominated convergence. Thus, since $H$ is compact, $\sup_{\gamma \in H}\vert \tilde{p}_i(\gamma, \theta_d) - \tilde{p}_i(\gamma, \theta_d^0) \vert$ is continuous in $\theta_d$. That is, we can make $\sup_{\gamma \in H}\vert \tilde{p}_i(\gamma, \theta_d) - \tilde{p}_i(\gamma, \theta_d^0) \vert$ arbitrarily small on $A_2 = A_2(\zeta, \veps)$ by picking $\zeta$ small enough, which is what we wanted to show. We next work with the second average in \eqref{eq:KL_bound_2}.

 By the mean value theorem, for some $\tilde{\beta}_{2, i}$ between $\beta_2$ and $\beta^0_2$, $\vert p_i(\beta_2^0, \theta_d^0) - p_i(\beta_2, \theta_d^0)\vert = \vert \E(c''(x_{i, i}^\tsp \tilde{\beta}_{2, i} + U^{(2)}_i + U^{(2)}_j))x_{i, i}^\tsp(\beta_2 - \beta_2^0)\vert$. Here, differentiation under the expectation is permissible since $c''$ is the variance of a Bernoulli random variable, hence bounded by $1/4$, and $ \vert x_{ii}^{\tsp}(\beta_2 - \beta_2^0)\vert \leq \Vert x_{i, i}\Vert \Vert \beta_2 - \beta_2^0\Vert^2 \leq \veps$ on $\bar{B}_\veps(\theta^0)$. By the same bound on $c''$ we get that $\E(c''(\gamma + U^{(1)}_i + U^{(2)}_j))$ is continuous in $\gamma$. Thus, $\inf_{\gamma \in H} \E(c''(\gamma + U^{(1)}_i + U^{(2)}_j)) \geq c_1 > 0$. That $c_1$ must be positive follows from that $c''$ is strictly positive on all of $\mathbb R$. We have thus proven that $\vert p_i(\beta_2^0, \theta_d^0) - p_i(\beta_2, \theta_d^0)\vert \geq c_1 \vert x_i^\tsp(\beta_2 - \beta_2^0)\vert$, uniformly on $\bar{B}_\veps(\theta^0)$. Using this and that $\vert x_{i, i}^\tsp(\beta_2 - \beta_2^0)\vert \leq \Vert x_{i, i}\Vert \Vert \beta_2 - \beta_2^0\Vert \leq \veps \leq 1$ so that squaring it makes it smaller,
 \begin{align*}
  N^{-1}\sum_{i = 1}^N\vert p_i(\beta_2^0, \theta_d^0) - p_i(\beta_2, \theta_d^0) \vert & \geq c_1 N^{-1}\sum_{i = 1}^N \vert x_{i, i}^\tsp (\beta_2 - \beta_2^0)\vert\\
  & \geq c_1N^{-1}  (\beta_2 - \beta_2^0)^\tsp \left(\sum_{i = 1}^N x_{i, i }x_{i, i}^\tsp\right) (\beta_2 - \beta_2^0) \\
  & \geq c_1 \Vert \beta_2 - \beta_2^0\Vert^2 N^{-1}\emin\left(\sum_{i = 1}^N x_{i, i }x_{i, i}^\tsp\right)
 \end{align*}
 which lower limit as $N \to \infty$ is bounded below by some strictly positive constant, say $c_2$, since $\liminf_{N \to \infty} N^{-1}\emin\left( \sum_{i = 1}^N x_{i, i}x_{i, i}^\tsp\right) \geq c_3 > 0$, for some $c_3$, and $\Vert \beta_2 - \beta_2^0\Vert \geq \veps / 2 > 0 $ on $A_2$. To summarize, we may pick $\zeta$ so small that the second average in \eqref{eq:KL_bound_2} is less than $ c_2/2$, say, and hence get $\sup_{\theta \in A_2}N^{-1} \E[\Lambda_N(\theta; W^{(2)})] \leq -2 (c_2 - c_2/2)^2 < 0$, for all but at most finitely many $N$. This proves point 1 as it pertains to $A_2$.

 Consider next $$A_1 = \partial B_\veps(\theta^0) \cap \left( \{\theta: \vert \theta_d - \theta_d^0\vert \geq \zeta\} \cup \{\theta: \Vert\beta_2 - \beta_2^0\Vert \leq \veps/2\}\right)$$
 and $W^{(1)}$. Similarly to for $W^{(2)}$, $\E[\Lambda_N(\theta; W^{(1)})]$ can due to independence be written as a sum of $N$ terms in the form
 \begin{align} \label{eq:norm_KL_neg}
 -\frac{1}{2}\left[\log\left(\frac{1 + 2\theta_d}{1 + 2\theta_d^0}\right) + \frac{1 + 2\theta_d^0 + [x_i^{\tsp}(\beta_2 - \beta_2^0)]^2}{1 + 2\theta_d}  - 1 \right],
 \end{align}
 which is the negative K--L divergence between two univariate normal distributions. Let us consider the possible values this can take for $\theta \in A_1$. If $\vert \theta_d - \theta_d^0\vert \geq \zeta$, then \eqref{eq:norm_KL_neg} is upper bounded by what is obtained when $\beta_1 = \beta_1^0$. This in turn is a continuous function in $\theta_d$ and hence attains its supremum on the compact set $\{\theta_d : \zeta \leq \vert \theta_d - \theta_d^0\vert \leq \veps\}$, and hence on $A_1$. This supremum is strictly positive because the divergence can be zero only if $\theta_d = \theta_d^0$. If instead $\Vert \beta_2 - \beta_2^0\Vert \leq \veps / 2$. Then either $\vert \theta_d - \theta_d^0\vert \geq \veps / 4$ or $\Vert \beta_1 - \beta_1^0\Vert \geq \veps / 4$, for otherwise it cannot be that $\Vert \theta - \theta^0\Vert = \veps$. If $\vert \theta_d - \theta_d^0\vert \geq \veps / 4$ the divergence in \eqref{eq:norm_KL_neg} has a lower bound away from zero by the same argument as for the cases $\vert \theta_d - \theta_d^0\vert \geq \zeta$. It remains to deal with the case $\Vert \beta_1 - \beta_1^0\Vert \geq \veps / 4$.

Writing \[[x_{i, i}^\tsp (\beta_1^0 - \beta_1)]^2 = (\beta_1^0 - \beta_1)^\tsp x_ix_i^\tsp (\beta_1^0 - \beta_1)\] we see that $-2 N^{-1} \Lambda_N(\theta; W^{(1)})$ is equal to
 \begin{align*}
     \log\left(\frac{1 + 2\theta_d}{1 + 2\theta_d^0}\right) +
     \frac{1 + 2\theta_d^0 + N^{-1} \sum_{i = 1}^N (\beta_1^0 - \beta_1)^\tsp x_ix_i^\tsp (\beta_1^0 - \beta_1)}{1 + 2\theta_d}  - 1,
 \end{align*}
 which has a lower limit that is greater than $$\log\left(\frac{1 + 2\theta_d}{1 + 2\theta_d^0}\right) + \frac{1 + 2\theta_d^0 + c_3(\veps / 4)^2}{1 + 2\theta_d}  - 1.$$
 This expression is in turn maximized in $\theta_d$ at $\theta_d = \theta_d^0 + c_3(\veps / 16)^2$; this follows from a straightforward optimization in $1 + 2\theta_d.$ The corresponding maximum evaluates to $\log(1 + 2\theta^0_d + c_3(\veps/4)^2) - \log(1 + 2\theta_d^0) > 0$. This finishes the proof of point 1.

 The proof of point 2 consists of checking the conditions of Lemma \ref{lem:pollard}. We first work with $A_1$ and $W^{(1)}$. Let $h_i(\omega, \theta) = \log[f_{\theta}(Y_{i, i, 1}(\omega)) / f_{\theta^0}(Y_{i, i, 1}(\omega))]$ be the log-likelihood ratio for the $i$th observation in the first subcollection, $i = 1, \dots, N$. We equip $\mathcal H_{N, \omega}$ with the $L_1$ norm $\Vert\cdot \Vert_1$, and $\Theta$ is equipped with the $L_2$ norm as before. To facilitate checking the two conditions we will first derive envelopes with the following properties: $\sup_{-\infty < i < \infty}\E H_i^k < \infty$ for every $k \geq 0$, $\sup_{-\infty < i < \infty}\pr(H_i \geq K) \to 0$ as $K \to 0$, and each $h_i(\omega, \theta)$ is $H_i$-Lipschitz in $\theta$ on $\bar{B}_\veps(\theta^0)$, and hence on $A_1$, for every $\omega$. We start with the Lipschitz property.

 Let us use the slight abuse of notation that $y_{i, i, 1} = Y_{i, i, 1}(\omega)$. Since the distribution of $W^{(1)}$ does not depend on $\beta_2$ we have $\nabla_{\beta_2}h_i(\omega, \theta) = 0$, and for some $c_1, c_2, c_3, c_4, c_5 > 0$ (depending on $\veps$), and every $\theta \in \bar{B}_\veps(\theta^0)$,
 \begin{align*}
  \Vert \nabla_{\beta_1}h_i(\omega, \theta) \Vert & = \Vert (y_{i, i, 1} - x_{i, i}^\tsp \beta_1)x_{i, i} / (1 + 2\theta_d)\Vert \leq c_1 \vert y_{i, i, 1}\vert +  c_2 \\
  \vert \nabla_{\theta_d}h_i(\omega, \theta)\vert & =\frac{1}{2}\left \vert \frac{1}{1 + 2\theta_d}  - (y_{i, i, 1} - x_{i, i}^\tsp \beta_1)^2/(1 + 2\theta_d)^2\right \vert \\
  & \leq c_3 + c_4 (\vert y_{i, i, 1}\vert + c_5)^2.
 \end{align*}
 Let $H_i$ be the sum of the bounds, i.e.
 \begin{align*}
  H_i(\omega) & = c_1 \vert y_{i, i, 1}\vert +  c_2  + c_3 + c_4 (\vert y_{i, i, 1}\vert + c_5)^2.
 \end{align*}
 By the mean value theorem, $\vert h_i(\omega, \theta) - h_i(\theta', \omega)\vert = \vert (\theta - \theta')^\tsp \nabla h_i(\omega, \tilde{\theta})\vert \leq  \Vert \theta - \theta'\Vert H_i$ for some $\tilde{\theta}$ between $\theta$ and $\theta'$. That is, $h_i$ is $H_i$-Lipschitz on $ \bar{B}_\veps(\theta^0)$. That $H_i$ is an envelope for $h_i$ follows from noting that $h_{i}(\omega, \theta^0) = 0$ so by taking $\theta' = \theta^0$ in the previous calculation, $\vert h_i(\omega, \theta)\vert \leq H_i \Vert \theta - \theta^0\Vert \leq H_i$ on $\bar{B}_\veps(\theta^0)$. That $\sup_{i} \E (H_i^k) < \infty$ for every $k > 0$ and $\sup_i \pr(H_i > K) \to 0$ as $K \to \infty$ follow from that $Y_{i, i, 1}$ is normally distributed with variance $1 + 2\theta_d^0$, not depending on $i$, and mean satisfying $-\Vert \beta_1^0\Vert \leq x_{i, i}^\tsp \beta_1^0 \leq \Vert \beta_1^0\Vert$. We are now ready to check the conditions of Lemma \ref{lem:pollard}.

 By the Cauchy--Schwartz inequality and the properties just derived, we have for every fixed $N$ that
 \begin{align*}
  N^{-1}\sum_{i = 1}^N \E [H_i I(H_i > K)] \leq \sup_{i}\E [H_i^2 ] \sup_{i}\pr(H_i \geq K) \to 0,\  K \to \infty,
 \end{align*}
 which verifies the first condition.

 For the second condition, note that the derived Lipschitz property gives, for arbitrary $h = (h_1(\omega, \theta), \dots, h_N(\omega, \theta))$ and $h' = (h_1(\omega, \theta'), \dots, h_N(\omega, \theta'))$ in $\mathcal H_{N, \omega}$:
 \begin{align*}
  \Vert h - h'\Vert_1 & = \sum_{i = 1}^N \vert h_{i}(\omega, \theta) - h_i(\omega, \theta')\vert \\
  & = \Vert \theta  - \theta'\Vert \Vert H\Vert_1.
 \end{align*}
 Thus, if we cover $\partial B_\veps(\theta^0)$ with $\epsilon$-balls with centers $\theta^j$, $j = 1, \dots, M$, then the corresponding $L_1$ balls in $\mathbb R^N$ of radius $\epsilon \Vert H\Vert_1$ with centers \[h^j = (h_1(\omega, \theta^j), \dots, h_N(\omega, \theta^j))\] cover $\mathcal{H}_{N, \omega}$. This is so because for every $\theta \in \partial B_\veps(\theta^0)$ there is a $\theta^j$ such that $\Vert \theta- \theta^j\Vert \leq \epsilon$, and hence by the Lipschitz property $\Vert h(\omega, \theta) - h(\omega, \theta^j)\Vert_1 \leq \Vert H\Vert_1 \epsilon$. Thus, $\mathrm{C}(\epsilon \Vert H\Vert_1, \mathcal H_{N, \omega}, \Vert \cdot \Vert_1) \leq \mathrm{C}(\epsilon, \partial B_\veps(\theta^0), \Vert \cdot \Vert)$. Since the covering number $\mathrm{C}(\epsilon, \partial B_\veps(\theta^0), \Vert \cdot \Vert)$ is constant in $N$, the second condition of Lemma \ref{lem:pollard} is verified for $A_1$ and $W^{(1)}$.

 The arguments for $A_2$ and $W^{(2)}$ are similar, redefining $h_i(\omega, \theta)$ with $Y_{i, i, 1}$ replaced by $Y_{1, 1, 2}$, taking $A_2$ in place of $A_1$, and so on. We need only prove the existence of envelopes $H_1, \dots, H_N$ with the desired properties. Using that $\vert y_{i, j, 2} - c'(\eta_{i, 2, 1})]\vert \leq 1$ and that $f_\theta(y_{i, i, 2}\mid u)f_\theta(u)/f_\theta(y_{i, i, 2}) = f_\theta(u\mid y_{i, i, 2})$ one gets,
 \begin{align*}
  \Vert \nabla_{\beta_2} h_i(\omega, \theta) \Vert & = \left \Vert\frac{1}{f_\theta(y_{i, i, 2})}\int f_\theta(y_{i, i, 2}\mid u)f_\theta(u)[y_{i, i, 2} - c'(\eta_{i, j, 2})]x_{i, i} \dd u\right \Vert \\
  & \leq \Vert x_{i, i}\Vert \leq 1.
 \end{align*}
 Using that $U^{(1)}_i$ and $U^{(2)}_j$ are the only random effects entering the linear predictor $\eta_{i, j, 2}$, and that $f_\theta(y_{i, j, 2}\mid u) \leq 1$,
 \begin{align*}
  \vert \nabla_{\theta_d} h_i(\omega, \theta)\vert & \leq \frac{1}{2 \theta_d f_\theta(y_{i, i, 2})} \int f_\theta(u^{(1)}_i, u^{(2)}_j)\left(\frac{(u^{(1)}_i)^2 + (u^{(2)}_j)^2}{\theta_d}\right) \dd u +\frac{1}{\theta_d}\\
  & = \frac{1}{\theta_d f_\theta(y_{i, j, 2})} + \frac{1}{\theta_d}.
 \end{align*}
 Due to continuity and compactness, the quantity in the last line attains its supremum on $\bar{B}_\veps(\theta^0)$. This maximum is finite for both $y_{i, i, 2} = 1$ and $y_{i, i, 2} = 0$ since the marginal success probability cannot be one or zero on interior points of $\Theta$.
 Thus, on $\bar{B}_\veps(\theta^0)$, $\Vert \nabla h_i(\omega, \theta) \Vert$ is bounded by a constant, say $H$, the largest needed for the two cases $y_{i, i, 2} = 0$ and $y_{i, i, 2} = 1$. By setting $H_i = H, i = 1, \dots, N$, we have envelopes with the right properties and this completes the proof of point 2.

Finally, we prove point 3. Consider without loss of generality the first subset and subcollection. For economical notation we omit dependence on the subcollection and write $L_N(\theta) = L_N(\theta; W^{(1)})$ and $\Lambda_N(\theta) = \Lambda_N(\theta; W^{(1)})$. Point 1 gives that $\sup_{\theta\in A_1}\E[\Lambda_N(\theta)] < -3\epsilon$ for some $\epsilon > 0$ and all large enough $N$. Assuming that $N$ is large enough that this holds, we get
 \begin{align*}
  \pr \left(e^{\epsilon N} \sup_{\theta \in A_1}L_N(\theta) > e^{-\epsilon N} \right)
  & \leq \pr\left(N^{-1}\sup_{\theta\in A_1}\Lambda_N(\theta) > \epsilon + \sup_{\theta \in A_1}\E[\Lambda_N(\theta)] \right) \\
  & \leq \pr\left(N^{-1}\sup_{\theta\in A_1}\left \vert \Lambda_N(\theta) - \E[\Lambda_N(\theta)]\right \vert > \epsilon \right),
 \end{align*}
 which vanishes as $N \to \infty$ by point 2.
\end{proof}

\end{document}